\documentclass{amsart}

\usepackage{enumerate}

\numberwithin{equation}{section}

\newcommand{\bC}{\mathbb{C}}

\newcommand{\bE}{\mathbb{E}}

\newcommand{\bP}{\mathbb{P}}
\newcommand{\bQ}{\mathbb{Q}}
\newcommand{\bR}{\mathbb{R}}

\newcommand{\bZ}{\mathbb{Z}}

\newcommand{\cB}{\mathcal{B}}
\newcommand{\cC}{\mathcal{C}}

\newcommand{\cE}{\mathcal{E}}
\newcommand{\cF}{\mathcal{F}}
\newcommand{\cG}{\mathcal{G}}
\newcommand{\cH}{\mathcal{H}}
\newcommand{\cI}{\mathcal{I}}

\newcommand{\cK}{\mathcal{K}}

\newcommand{\cN}{\mathcal{N}}

\newcommand{\cS}{\mathcal{S}}

\theoremstyle{plain} \newtheorem{Theo}{Theorem}[section]
\theoremstyle{plain} \newtheorem{Lemma}[Theo]{Lemma}
\theoremstyle{plain} 
\theoremstyle{plain} 
\theoremstyle{remark} \newtheorem{Rem}[Theo]{Remark}
\theoremstyle{definition} 
\theoremstyle{definition} 
\theoremstyle{remark}

\begin{document}

\title[A zero--one law for linear transformations]{A  zero--one law 
for \\
linear transformations \\
of L\'evy noise}

\author{Steven N. Evans}

\email{evans@stat.Berkeley.EDU}

\address{Department of Statistics \#3860 \\
 University of California at Berkeley \\
367 Evans Hall \\
Berkeley, CA 94720-3860 \\
U.S.A}

\thanks{Research supported in part by NSF grants DMS-0405778 and DMS-0907630}

\date{\today}

\keywords{Hewitt--Savage, white noise, Poisson noise, 
special linear group, orthogonal group, 
Auerbach problem, Jordan canonical form}

\subjclass[2000]{Primary:  60F20, 28D15, 60G57; Secondary: 15A21, 60G55, 60H40.}

\begin{abstract}
A L\'evy noise on $\bR^d$ assigns 
a random real ``mass'' $\Pi(B)$ to each Borel subset $B$ of $\bR^d$ with finite Lebesgue measure. 
The distribution of $\Pi(B)$ only depends on the Lebesgue measure of $B$, and
if $B_1, \ldots, B_n$ is a finite collection of pairwise disjoint sets, 
then the random variables $\Pi(B_1), \ldots, \Pi(B_n)$ are independent
with $\Pi(B_1 \cup \cdots \cup B_n) = \Pi(B_1) + \cdots + \Pi(B_n)$
almost surely.  In particular, the distribution of $\Pi \circ g$ is the same
as that of $\Pi$ when $g$ is a bijective transformation of $\bR^d$ that preserves Lebesgue
measure. It follows from the Hewitt--Savage zero--one law that any event which is
almost surely invariant under the mappings $\Pi \mapsto \Pi \circ g$ for
every Lebesgue measure preserving bijection $g$ of $\bR^d$ must have
probability $0$ or $1$.  We investigate whether certain smaller groups of
Lebesgue measure preserving bijections also possess this property.  We show
that if $d \ge 2$, the L\'evy noise is not purely deterministic, 
and the group consists of linear transformations and is closed, then the 
invariant events all have probability $0$ or $1$ if and only if the group is not compact.
\end{abstract}

\maketitle

\section{Introduction}

The zero-one law of Hewitt and Savage \cite{MR0076206} concerns sequences of
of independent, identically distributed, random variables $X = \{X_k : k \in \bZ\}$
on some probability space $(\Omega, \cF, \bP)$.
It says that if $A \subseteq \bR^\bZ$ is any product measurable set such that
$g A$ and $A$ differ by a $\bP$-null set
for all bijections $g : \bZ \rightarrow \bZ$, then
$\bP\{X \in A\}$ is either $0$ or $1$.  Of course, it is not important
that $X$ is indexed by $\bZ$: we could replace $\bZ$ by any countable set.

One natural family of continuous analogues of the family of sequences
of independent, identically distributed, random variables
is the family the L\'evy noises.   Recall that
a L\'evy noise on $\bR^d$ is defined as follows.
Let $\mu$ be an infinitely divisible probability measure on $\bR$.
There is an associated convolution semigroup $(\mu_t)_{t \ge 0}$ of probability
measures on $\bR$: that is, 
\begin{itemize}
\item
$\mu_1$ is $\mu$
\item 
$\mu_0$ is $\delta_0$, the point mass at $0$,
\item
$\mu_s \ast \mu_t = \mu_{s+t}$, for all $s,t \ge 0$, where
$\ast$ denotes convolution, 
\item
the weak limit as $t \downarrow s$ of
$\mu_t$ is $\mu_s$ for all $s \ge 0$.
\end{itemize}
Denote the Borel $\sigma$-field of $\bR^d$ by $\cB(\bR^d)$.
Write $\Lambda$ for Lebesgue measure
on $\bR^d$ and let $\cC(\bR^d)$ be the subset of
$\cB(\bR^d)$ consisting of sets with finite Lebesgue measure.
A L\'evy noise on $\bR^d$ corresponding to $\mu$ is
a collection of real-valued random variables
$\Pi = \{\Pi(B) : B \in \cC(\bR^d)\}$ on some probability
space $(\Omega, \cF, \bP)$ with the properties:
\begin{itemize}
\item
the random variable
$\Pi(B)$ has distribution $\mu_{\Lambda(B)}$ for all $B \in \cC(\bR^d)$,
\item
if $B_1, \ldots, B_n$ is a finite collection of pairwise disjoint sets
in $\cC(\bR^d)$, then the random variables
$\Pi(B_1), \ldots, \Pi(B_n)$ are independent
and $\Pi(B_1 \cup \cdots \cup B_n) = \Pi(B_1) + \cdots + \Pi(B_n)$
almost surely.  
\end{itemize}
For each infinitely divisible probability measure $\mu$ 
it is possible to construct (via Kolmogorov's extension theorem)
a corresponding L\'evy noise on $\bR^d$ for every $d$.  Note that
if $\mu$ is not a point mass, then the
random variable $\Pi(B)$ is not almost surely
constant when $B \in \cC(\bR^d)$ is a set with $\Lambda(B)>0$.

The most familiar examples of L\'evy noises are the usual
{\em Gaussian white noise}, in which case $\mu$
is the standard Gaussian probability distribution, and the
{\em homogeneous Poisson random measures}, in which case
$\mu$ is a Poisson distribution with some positive
mean.

Let $\Sigma$ be the Cartesian product $\bR^{\cC(\bR^d)}$ and
write $\cS$ for the corresponding product $\sigma$-field.  The L\'evy noise
$\Pi$ is a measurable map from $(\Omega, \cF)$ to $(\Sigma, \cS)$.
Given a bijection $g : \bR^d \rightarrow \bR^d$ 
that is Borel measurable with a Borel measurable inverse,  
there is a corresponding 
bijection $T_g : \Sigma \rightarrow \Sigma$ that maps the element
$(\pi(B))_{B \in \bR^{\cC(\bR^d)}}$ to the element 
$(\pi(g^{-1} B))_{B \in \bR^{\cC(\bR^d)}}$.
The mapping $T_g$ and its inverse are both measurable.
Note that $T_g \circ \Pi$ has the same distribution as
$\Pi$ when $g$ preserves the Lebesgue measure $\Lambda$.

If $G$ is group of Lebesgue measure preserving bijections, then the corresponding 
{\em invariant $\sigma$-field} $\cI_G$ is the collection of sets  $S \in \cS$
with the property 
\[
\bP(\{\Pi \in S \} \triangle \{T_g \Pi \in S\}) = 0
\text{ for all }g \in G, 
\]
where $\triangle$ denotes the symmetric difference.

It follows readily from the Hewitt--Savage zero--one law that if
$G$ is the group of all Borel measurable bijections that have
Borel measurable inverses and preserve Lebesgue measure, 
then the invariant $\sigma$-field $\cI_G$ consists of events
with probability $0$ or $1$.  

However, the same conclusion still holds for much ``smaller'' groups $G$.
For example, it holds when $G$ is $\bR^d$ acting on itself via translations
(this follows from the multiparameter ergodic theorem 
and the Kolmogorov zero--one law).
On the other hand, the conclusion fails when $\mu$ is not a point mass
and $G$ is the group $\mathrm{0}(\bR^d)$
of linear transformations of $\bR^d$ that preserve the usual Euclidean
inner product (and hence also preserve Lebesgue measure); for example, in that
case the random variable $\Pi(\{x \in \bR^d : \|x\| \le 1\})$ is $\cI_G$-measurable
but it is not almost surely constant.

Our aim in this paper is to characterize the closed groups
of linear transformations of $\bR^d$ that preserve Lebesgue measure
and for which the corresponding invariant $\sigma$-field
consists of events with probability $0$ or $1$.

Recall that a linear mapping of $\bR^d$ into itself preserves Lebesgue measure
if and only if the corresponding matrix with respect to some basis
of $\bR^d$ has a determinant with absolute value $1$.  Of course,
if this condition holds for one basis, then it holds for all bases.
The collection of linear maps that preserve Lebesgue measure
is a group.  Denote this group by $\Gamma$.  We have
$\Gamma = (+1) \times \mathrm{Sl}(\bR^d) \sqcup (-1) \times \mathrm{Sl}(\bR^d)$,
where $\mathrm{Sl}(\bR^d)$ is the group of linear maps with
determinant $1$.  We will think of $\Gamma$ as either a group
of linear transformations or as a group of matrices.

Our main result is the following.

\begin{Theo}
\label{T:main}
Suppose that $d \ge 2$ and $\mu$ is not a point mass.
Let $G$ be a closed subgroup of $\Gamma$.
The corresponding invariant $\sigma$-field $\cI_G$ 
consists of sets with probability $0$ or $1$
if and only if $G$ is not compact.
\end{Theo}

We prove Theorem~\ref{T:main} 
in Section~\ref{S:main_proof}
after some preparatory results
in Section~\ref{S:preparatory}.
The proof also uses some consequences of
the Jordan canonical form for matrices 
that are not similar to orthogonal matrices.
We establish the relevant results in Section~\ref{S:Jordan}.

\begin{Rem}
\label{R:compact_subgroups}
We note that a closed subgroup $G$ of $\Gamma$ is compact
if and only if there is an invertible matrix $h$ such that
$h^{-1} G h \subseteq \mathrm{0}(\bR^d)$, where
$\mathrm{0}(\bR^d)$ is the group of 
$d \times d$ orthogonal matrices.
This fact follows from general Lie group theory
and is well-known, but we have
found an explicit statement with a self-contained accompanying
proof to be somewhat elusive.
For the sake of completeness, we note the following
simple bare hands proof based on Weyl's ``unitarian trick''.  
Let $\eta$ be the normalized Haar measure on $G$. 
Define a real inner product $\langle \cdot, \cdot \rangle_\eta$
on $\bR^d$ by $\langle x , y \rangle_\eta := \int_G (g x)^\top (g y) \, \eta(dg)$,
where $u^\top$ denotes the transpose of the vector $u$.
It is clear that $\langle g x , g y \rangle_\eta = \langle x , y \rangle_\eta$
for any $g \in G$ and $x,y \in \bR^d$.  
There is a positive definite symmetric matrix $S$ such that 
$\langle x , y \rangle_\eta = x^\top S y$
(see Exercise 14 in Section 7.2 of \cite{MR1084815}).
Let $h = S^{-\frac{1}{2}}$ be the inverse of the usual positive definite symmetric
square root of $S$ (see Theorem 7.2.6 of \cite{MR1084815}).  Then, 
\[
\begin{split}
(h^{-1} g h x)^\top (h^{-1} g h y)
&=
x^\top h g^\top h^{-1} h^{-1} g h y
=
x^\top h g^\top S g h y \\
&=
\langle g h x , g h y \rangle_\eta
=
\langle h x, h y \rangle_\eta \\
&=
x^\top h S h y 
=
x^\top y. \\
\end{split}
\]
Thus, $h^{-1} g h$ preserves the usual Euclidean inner product
on $\bR^d$ and is an orthogonal matrix, as required.

Suppose that $G$ is compact and
$h$ is such that $h^{-1} G h$ consists of orthogonal matrices.  Let
$U$ be the closed unit ball in $\bR^d$ for the usual Euclidean
metric. Then, $g (h U) = (h U)$ for all $g \in G$.  Conversely,
suppose that $G$ is a closed subgroup of $\Gamma$ such that
$g K \subseteq K$ for all $g \in G$, where is
a compact set with $0$ in its interior.  It follows that
the $\ell^2$ operator norms of the elements of $G$ are bounded,
and hence $G$ is compact.
\end{Rem}

We finish this introduction with some comments about the motivations
that led us to consider the question we study in this paper.

A first motivation comes from the forthcoming paper
\cite{Hol_Lyo_Soo_09} on ``deterministic Poisson thinning''
that we heard about in a lecture by Omer Angel during the
2009 Seminar on Stochastic Processes held at Stanford University.

Let $M$ be the space of non-negative integer valued Radon measures
on $\bR^d$ for which all atoms are of mass $1$ (that is,
$M$ is the space of possible realizations of a {\em simple point
point process} on $\bR^d$).
Note that $M$ may be viewed as a subset of $\Sigma$.
Equip $M$ with the vague topology.
It is shown in \cite{Hol_Lyo_Soo_09} that for $0 < \alpha < \beta$ there is a
Borel measurable map $\Theta: M \rightarrow M$ such that
$\Theta(m) \le m$ for all $m \in M$ and 
if $\Pi$ is a homogeneous Poisson process with intensity $\beta$, then 
$\Theta(\Pi)$ is a homogeneous Poisson process with intensity $\alpha$.
Moreover, if $G$ is the group of affine Euclidean isometries of 
$\bR^d$, then $\Theta \circ T_g = T_g \circ \Theta$ for all $g \in G$.

It is natural to ask if this equivariance property can hold
for some larger group $G$ of affine Lebesgue measure preserving maps.  Suppose that
this is possible.  Take $\bP$ to be the distribution of the
homogeneous Poisson process with intensity $\beta$.  Write $\bP^x$,
$x \in \bR^d$, for the associated family of Palm distributions.
That is, $\bP^x$ is, heuristically speaking, the distribution
of a pick from $\bP$ conditioned to have an atom of mass $1$ at $x$.  
In this Poisson case, $\bP^x$ is, of course, just the distribution of
the random measure obtained by taking a pick from $\bP$ and
adding an extra atom at $x$.  It follows from the equivariance of
$\Theta$ under $G$ that if we let
$H$ be the subgroup of $G$ that fixes $0$, then the
map $\gamma: M \rightarrow \{0,1\}$ given by $\gamma(m) = (\Theta(m))(\{0\})$
has the property $\gamma \circ T_h = \gamma$, $\bP^0$-a.s.
for all $h \in H$, and $\bP^0\{\gamma = 1\} = \frac{\alpha}{\beta}$.  
Consequently, if we define $\epsilon: M \rightarrow \{0,1\}$
by $\epsilon(m) = \gamma(m + \delta_0)$, where $\delta_0$ is the unit point
mass at $0$, then $\epsilon \circ T_h = \epsilon$, $\bP$-a.s. for all $h \in H$,
and $\bP\{\epsilon = 1\} = \frac{\alpha}{\beta}$.

However, Theorem~\ref{T:main} says that this is impossible if
$H$ strictly contains the group $\mathrm{0}(\bR^d)$ of
linear Euclidean isometries.

A second motivation comes from an analogy with a result in
\cite{MR0451399}.  Suppose now that $\bP$ is the distribution of a simple point
process on $\bR^d$.  If $\bP$ is invariant for all the transformations
$T_g$, $g \in G$, where $G$ is the group of all bijections that preserve
Lebesgue measure, then it follows from de Finetti's theorem that
$\bP$ is of the form $\int \bQ^\alpha \, q(d\alpha)$, where
$\bQ^\alpha$ is the distribution of the homogeneous Poisson process
on $\bR^d$ with intensity $\alpha$ and the mixing measure
$q$ is a probability measure on the nonnegative real numbers.  This
result may be thought of as a continuum analogue of the special
case of de Finetti's theorem which says that an exchangeable sequence of
$\{0,1\}$ valued random variables is a mixture of independent, identically
distributed, Bernoulli sequences.
A counterexample is presented in \cite{MR0451399} (see also \cite{MR546881})
demonstrating that if $G$ is replaced by the smaller group of affine Lebesgue measure 
preserving transformations, then such a conclusion is false.  

In the same way that this result addresses continuum
analogues of de Finetti's theorem for small groups of
measure preserving transformations, it is natural to consider
whether there are continuum analogues of the Hewitt--Savage
zero--one law for such groups.

\section{Preparatory results}
\label{S:preparatory}

Without loss of generality, we may suppose from now on that 
$\Omega = \Sigma$,
$\Pi$ is the identity map,
and $\cF$ is the $\bP$-completion of $\cS$.  Write
$\cN$ for the sub-$\sigma$-field of $\cF$ consisting
of sets with probability $0$ or $1$.

Given $B \in \cB(\bR^d)$, set 
$\cF_B := \sigma\{\Pi(C) : C \in \cC(\bR^d), \, C \subseteq B\} \vee \cN$.
Note for $g \in G$ that if $\Psi : \Omega \rightarrow \bR$
is $\cF_B$-measurable, 
then $\Psi \circ T_{g^{-1}}$ is $\cF_{g B}$-measurable, and, moreover,
if $\Upsilon : \Omega \rightarrow \bR$ is $\cF_{g B}$-measurable,
then $\Upsilon = \Psi \circ T_{g^{-1}}$ for some $\cF_B$-measurable $\Psi$.
Note also that $\cF_{B'} \subseteq \cF_{B''}$ when $B' \subseteq B''$.

\begin{Lemma}
\label{L:equivariance}
Suppose that $\Phi: \Omega \rightarrow \bR_+$ is a bounded
$\cI_G$-measurable function.  Then, for $g \in G$ and $B \in \cB(\bR^d)$,
\[
\bE \left[\Phi \, | \, \cF_B \right]
=
\bE \left[\Phi \, | \,  \cF_{g B} \right] \circ T_{g}.
\]
Consequently, the distribution of 
$\bE \left[\Phi \, | \,  \cF_{g B} \right]$
does not depend on $g \in G$.
\end{Lemma}

\begin{proof}
 By the remarks prior to the
the statement of the lemma, 
$\bE \left[\Phi \, | \,  \cF_{g B} \right] \circ T_{g}$
is $\cF_B$-measurable.  Moreover,
if $\Psi : \Omega \rightarrow \bR_+$ is any bounded
$\cF_B$-measurable function, then
\[
\begin{split}
\bE \left[\Phi \times \Psi \right]
& =
\bE \left[\left(\Phi \circ T_{g^{-1}} \right)\times \left(\Psi \circ T_{g^{-1}}\right)\right] \\
& =
\bE \left[\Phi \times  \left(\Psi \circ T_{g^{-1}}\right) \right] \\
& =
\bE \left[ \bE \left[\Phi \, | \,  \cF_{g B} \right] \times \left(\Psi \circ T_{g^{-1}} \right)\right] \\
& =
\bE \left[\left( \bE \left[\Phi \, | \,  \cF_{g B} \right] \circ T_g \right) \times \left( \Psi \circ T_{g^{-1}} \circ T_g \right) \right] \\
& =
\bE \left[ \left( \bE \left[\Phi \, | \,  \cF_{g B} \right] \circ T_{g} \right) \times \Psi  \right], \\
\end{split}
\]
and so $\bE \left[\Phi \, | \,  \cF_{g B} \right] \circ T_{g}$ is 
$\bE \left[\Phi \, | \, \cF_B \right]$, as claimed.
\end{proof}

Denote by $\cK(\bR^d)$ the collection of compact subsets of $\bR^d$.

\begin{Lemma}
\label{L:compact_approximation}
For any $B \in \cB(\bR^d)$, the $\sigma$-fields $\cF_B$ and
$\sigma\{\Pi(C) : C \in \cK(\bR^d), \, C \subseteq B\} \vee \cN$ coincide.
\end{Lemma}

\begin{proof}
Suppose that $A \in \cC(\bR^d)$.  By the inner regularity
of Lebesgue measure, there exist compact sets
$C_1 \subseteq C_2 \subseteq \ldots \subseteq C$ such that
$\lim_{n \rightarrow \infty} \Lambda(C_n) = \Lambda(C)$.
We have $\Pi(C) = \Pi(C_n) + \Pi(C \setminus C_n)$
almost surely.  Also, $\Pi(C \setminus C_n)$
has distribution $\mu_{\ell_n}$, where $\ell_n = \Lambda(C \setminus C_n)$,
and so $\Pi(C \setminus C_n)$ converges to $0$ in probability
as $n \rightarrow \infty$.  Hence, there exists a subsequence $(n_k)$
such that $\Pi(C \setminus C_{n_k})$ converges to $0$ almost surely
as $k \rightarrow \infty$, so that $\Pi(C_{n_k})$ converges to $\Pi(C)$
almost surely.  The result follows directly from this observation.
\end{proof}

\begin{Lemma}
\label{L:Kolmogorov}
Suppose that $A_h \in \cB(\bR^d)$, $h \in \bZ$, is a family of sets with
the properties $A_{h'} \subseteq A_{h''}$ for $h' < h''$, 
$\Lambda\left(\bigcap_{h \in \bZ} A_h \right) = 0$, and
$\Lambda\left(\bR^d \setminus \bigcup_{h \in \bZ} A_h \right) = 0$.
Then,
$\bigcap_{h \in \bZ} \cF_{A_h} = \cN$ and
$\bigvee_{h \in \bZ} \cF_{A_h} = \cF$.
\end{Lemma}

\begin{proof}
Consider the claim regarding $\bigcap_{h \in \bZ} \cF_{A_h}$.
It suffices to show that if $\Psi$ is any bounded, non-negative,
$\cF$-measurable random variable, then 
$\bE\left[\Psi \, | \, \bigcap_{h \in \bZ} \cF_{A_h} \right]$
is almost surely constant.  By the reverse martingale
convergence theorem, the latter random variable is almost
surely $\lim_{h \rightarrow -\infty} \bE\left[\Psi \, | \,  \cF_{A_h} \right]$.

Set $B = \bR^d \backslash \bigcap_{h \in \bZ} A_h$.
Note for any $C \in \cC(\bR^d)$ that 
$\Pi(C) = \Pi(C \cap \bigcap_{h \in \bZ} A_h) + \Pi(C \cap B) = \Pi(C \cap B)$ almost surely because $\Lambda(C \cap \bigcap_{h \in \bZ} A_h)=0$,
and hence $\cF = \cF_B$.
Thus, by Lemma~\ref{L:compact_approximation}, 
$\cF = \sigma\{\Pi(C) : C \in \cK(\bR^d), \, C \subseteq B\} \vee \cN$.

Therefore, given any $\epsilon > 0$ there exist compact
subsets $C_1, \ldots, C_n$ of $B$ and a bounded
Borel function $F : \bR^n \rightarrow \bR_+$ such that
\[
\bE \left[ \left|\Psi - F(\Pi(C_1), \ldots, \Pi(C_n))\right| \right] < \epsilon,
\]
and so 
\[
\bE
\left[
\left|
\bE\left[\Psi \, | \,  \cF_{A_h} \right] 
- \bE\left[F(\Pi(C_1), \ldots, \Pi(C_n)) \, | \,  \cF_{A_h} \right]
\right|
\right]
< \epsilon
\]
for all $h \in \bZ$.  

When $h$ is sufficiently small, the compact sets 
$C_1, \ldots, C_n$ are all contained in the complement of $A_h$.
In that case, the random variable $F(\Pi(C_1), \ldots, \Pi(C_n))$
is independent of the $\sigma$-field $\cF_{A_h}$ and hence
$\bE\left[F(\Pi(C_1), \ldots, \Pi(C_n)) \, | \,  \cF_{A_h} \right]$
is almost surely constant.  Therefore, 
$\bE\left[\Psi \, | \, \bigcap_{h \in \bZ} \cF_{A_h} \right]$
is within $L^1(\bP)$ distance $\epsilon$
of a constant for all $\epsilon > 0$
and so this random variable is itself almost surely constant,
as required.

The claim regarding $\bigvee_{h \in \bZ} \cF_{A_h}$ can be established
similarly, and we leave the proof to the reader.
\end{proof}

\section{Proof of Theorem~\ref{T:main}}
\label{S:main_proof}

Suppose that the group $G$ is compact.  
By Remark~\ref{R:compact_subgroups}, there is an invertible matrix
$h$ such that $g (h U) = (h U)$ for all $g \in G$,
where $U$ is the closed unit ball around $0$ in $\bR^d$ 
for the usual Euclidean metric. The random variable
$\Pi(hU)$ is $\cI_G$-measurable and, 
by the assumption on $\mu$, has distribution
$\mu_{\Lambda(hU)}$ that is not concentrated at a point.  Therefore, $\cI_G$
contains sets that have probability strictly between $0$ and $1$.

Conversely, suppose that the closed group $G$ is not compact.   
Then, by Theorem 1 of
\cite{MR0197577},
there is matrix $g \in G$ such that the cyclic group $\{g^h : h \in \bZ\}$
does not have a compact closure. 
We note that this result is non-trivial
and is related to the ``Auerbach problem'' 
-- see also \cite{MR0120127, MR0191997}.

Let $(D_t)_{0 < t < \infty}$ be the corresponding
increasing family of closed subsets of $\bR^d$ guaranteed
by Lemma~\ref{L:shrinking_sets} below.  Set $\cG_t = \cF_{D_t}$.
Because 
$\Lambda\left(\bR^d \setminus \bigcup_{0 < t < \infty} D_t\right) = 0$, 
it follows from Lemma~\ref{L:Kolmogorov}
that 
$\cF = \bigvee_{0 < t < \infty} \cG_t$.

Suppose, contrary to the statement of the theorem,
that there is a bounded $\cI_G$-measurable function
$\Phi: \Omega \rightarrow \bR_+$  that is not almost surely
equal to a constant.  By the martingale
convergence theorem,
\[
\Phi 
= 
\bE\left[\Phi \, \bigg | \, \bigvee_{n=1}^\infty \cG_n \right]
=
\lim_{n \rightarrow \infty} \bE\left[\Phi \, | \, \cG_n \right], 
\quad \text{$\bP$-a.s.},
\]
where the limit is taken over the positive integers.  Consequently,
there is a positive integer $N$ such that 
$\bE\left[\Phi \, | \, \cG_N \right]$ is not $\bP$-almost surely
equal to a constant.  In particular, the variance of
$\bE\left[\Phi \, | \, \cG_N \right]$ is strictly positive.

Because $\Lambda\left(\bigcap_{0 < t < \infty} D_t\right) = 0$,
it follows from Lemma~\ref{L:Kolmogorov} that
$\bigcap_{0 < t < \infty} \cG_t = \cN$.
For a positive integer $n$, set $\cH_n = \cG_{\frac{N}{n}}$.
Note that $\cH_1 \supseteq \cH_2 \supseteq \ldots$
and $\bigcap_{n=1}^\infty \cH_n = \cN$.
By the reverse martingale convergence theorem, 
\[
\lim_{n \rightarrow \infty} \bE\left[\Phi \, \bigg | \, \cH_n \right]
=
\bE\left[\Phi \, \bigg | \, \bigcap_{n=1}^\infty \cH_n \right]
=
\bE\left[\Phi\right], 
\quad \text{$\bP$-a.s. and in $L^2(\bP)$}.
\]
In particular, the variance of $\bE\left[\Phi \, | \, \cH_n \right]$
converges to $0$ as $n \rightarrow \infty$.

For a non-negative integer $m$, set $\cE_m = \cF_{g^m D_N}$. Thus,
$\cE_0 = \cG_N = \cH_1$.
It follows from Lemma~\ref{L:equivariance} that the distribution of
$\bE\left[\Phi \, | \, \cE_m \right]$ is 
that of $\bE\left[\Phi \, | \, \cG_N \right]$ for all  $m$.
In particular, the variance of 
$\bE\left[\Phi \, | \, \cE_m \right]$ 
is the same as that of 
$\bE\left[\Phi \, | \, \cG_N \right]$ for all $m$.
Because $g^h D_{t''} \subseteq D_{t'}$ for 
$0 < t' < t'' < \infty$ and $h$ sufficiently large, we have 
for any given positive integer $n$ that there exists an integer $m$
for which  $\cE_m \subseteq \cH_n$.  In that case,
\[
\bE\left[\bE\left[\Phi \, | \, \cE_m \right] \right] 
= \bE[\Phi] 
= \bE\left[\bE\left[\Phi \, | \, \cH_n \right] \right]
\]
and, by the conditional Jensen's inequality,
\[
\bE\left[\bE\left[\Phi \, | \, \cE_m \right]^2 \right] 
\le
\bE\left[\bE\left[\Phi \, | \, \cH_n \right]^2 \right]
\]
so that the variance of $\bE\left[\Phi \, | \, \cE_m \right]$
is dominated by the variance of $\bE\left[\Phi \, | \, \cH_n \right]$.

The former variance does not depend on $m$ and is strictly positive,
whereas the latter variance converges to $0$ as $n \rightarrow \infty$,
so we arrive at a contradiction.

\section{Consequences of the Jordan canonical form}
\label{S:Jordan}

Let $A$ be a $d \times d$ matrix with entries from the field $\bC$
of complex numbers.  For convenience, we say that $A$
has {\em order} $d$.
We recall some facts from linear algebra that may be found, for example,
in Ch 3 of \cite{MR1084815}.

The {\em geometric multiplicity}
of an eigenvalue $\lambda$ of $A$ is the dimension of the
null space of the matrix $\lambda I - A $ 
(that is, the the geometric multiplicity is the maximal number of linearly
independent solutions of the equation $A x = \lambda x$, $x \in \bC^d$).
The {\em algebraic multiplicity} of the eigenvalue $\lambda$ is
the multiplicity of $\lambda$ as a root of the characteristic
equation $t \mapsto \det(t I - A)$ (that is, the algebraic multiplicity
is the largest positive
integer $m$ such that the polynomial $(t - \lambda)^m$
divides the polynomial $\det(t I - A)$).  

Suppose that the sum of the geometric multiplicities of the eigenvalues
of $A$ is $k$.  Because eigenvalues corresponding to distinct eigenvalues
are linearly independent, $k$ is the dimension of the sum of the
null spaces of $\lambda I - A$ as $\lambda$ ranges over the eigenvalues of $A$.

For a positive integer $r$ and $\zeta \in \bC$, let
$J_r(\zeta)$ be the $r \times r$ matrix given by
\[
J_r(\zeta)_{ij} :=
\begin{cases}
\zeta, & i=j,\\
1, & j=i+1,\\
0, & \text{otherwise}.
\end{cases}
\]
That is, every entry of $J_r(\zeta)$ on the diagonal is $\zeta$, 
every entry on the super-diagonal is $1$, and every other entry is $0$.

There exists an invertible matrix $S$ with entries from $\bC$ such that
$J := S^{-1} A S$ is block diagonal with blocks 
$J_{d_1}(\lambda_1), \ldots, J_{d_k}(\lambda_k)$.  The numbers
$\lambda_1, \ldots, \lambda_k$ are all eigenvalues of $A$, with each distinct
eigenvalue appearing at least once.  The geometric multiplicity of
an eigenvalue $\lambda$ is the number of times that $\lambda$ appears in the
list $\lambda_1, \ldots, \lambda_k$.  The algebraic multiplicity of
$\lambda$ is the sum of the orders of the corresponding blocks.
The matrix $J$, which is unique up to a re-ordering of the $\lambda_p$,
is the {\em complex Jordan canonical form} of $A$.

An order $2r$ matrix of the form 
\[
\begin{pmatrix}
J_r(\zeta) & 0 \\
0 & J_r(\bar \zeta) 
\end{pmatrix},
\]
where $\bar \zeta = c - id$ is the complex conjugate of $\zeta = c + id$, is similar
to a block matrix $C_r(\zeta)$ in which each block has order $2$, the
diagonal blocks are all of the form
\[
\begin{pmatrix}
c & d \\
-d & c 
\end{pmatrix},
\]
the super-diagonal blocks are all  identity matrices, and the
remaining blocks are all $0$ matrices..  

Suppose now that the entries of $A$ are in $\bR$.  Define $\lambda_1, \ldots,
\lambda_k$ and $J$ as before.  If some $\lambda_p$ is not real, then
its complex conjugate $\bar \lambda_p$ appears as $\lambda_q$ for some
$q$ with $d_q = d_p$.  There is an invertible matrix 
$T$ with entries from $\bR$ such that
$K := T^{-1} A T$ is block diagonal with blocks 
$J_{a_1}(\eta_1), \ldots, J_{a_s}(\eta_s), 
C_{b_1}(\kappa_1), \ldots, C_{b_t}(\kappa_t)$.  The numbers
$\eta_1, \ldots, \eta_s$ are the real eigenvalues in the list
$\lambda_1, \ldots, \lambda_k$ while the numbers $\kappa_1, \ldots \kappa_t$
come from picking one member of each complex conjugate pair
of non-real eigenvalues in the list. If $\eta_m = \lambda_\ell$, then
$a_m = d_\ell$, and if $\kappa_n \in \{\lambda_p, \lambda_q\}$
with $d_p = d_q$, then $b_n =  d_p =  d_q$.  
The matrix $K$ is the {\em real Jordan canonical form} of $A$.

Suppose now that $A \in \Gamma = (+1) \times \mathrm{Sl}(\bR^d) \sqcup (-1) \times \mathrm{Sl}(\bR^d)$,
so that $\det A = \lambda_1^{d_1} \cdots \lambda_k^{d_k} \in \{\pm 1\}$.
The matrix $A$ is similar to an orthogonal matrix if and only if $k=d$
(equivalently, $d_1 = \cdots = d_k = 1$) and $|\lambda_1| = \cdots = |\lambda_k| = 1$.
Also,  $A$ is similar to an orthogonal matrix if and only if the set of matrices 
$\{A^h : h \in \bZ\}$ is bounded.
Hence, if $\{A^h : h \in \bZ\}$ is not bounded, then
 $A$ is not similar to an orthogonal matrix and either $|\lambda_\ell| < 1$
for some $\ell$ or $|\lambda_1| = \cdots = |\lambda_k| = 1$ and $d_\ell \ge 2$
for some $\ell$.  It follows that one or more of the blocks in the real Jordan
canonical form of $A$ must have one of the following forms:
\begin{itemize}
\item[a)]
$J_a(\eta)$ with $a \ge 2$ and $\eta \in \{\pm 1\}$,
\item[b)]
$C_b(\kappa)$ with $b \ge 2$ and $\eta \in \{z \in \bC \setminus \bR : |z| = 1\}$,
\item[c)]
$J_a(\eta)$ with $a \ge 1$ and $\eta \in \{x \in  \bR : |x| < 1\}$,
\item[d)]
$C_b(\kappa)$ with $b \ge 1$ and $\eta \in \{z \in \bC \setminus \bR : |z| < 1\}$.
\end{itemize}

Consider case (a).  Note that $J_a(\eta) = \eta I + N_a$, where $N_a := J_a(0)$
is the order $a$ matrix with $1$ at every position on the super-diagonal
and $0$ elsewhere.  Write $e_1, \ldots, e_a$ for the standard basis of column
vectors of $\bR^d$; that is, $e_i$ has $1$ in the $i^{\mathrm{th}}$ coordinate
and $0$ elsewhere.  Observe that 
\[
N_a^j e_i = 
\begin{cases}
e_{i-j}, & \text{if $j < i$}, \\
0, & \text{otherwise}.
\end{cases}
\]
Thus,
\[
\begin{split}
& J_a(\eta)^h(x_1 e_1 + \cdots + x_a e_a) \\
& \quad =
\sum_{i=1}^a x_i \sum_{j=0}^h \binom{h}{j} \eta^{h-j} N_a^j e_i \\
& \quad =
\sum_{i=1}^a x_i \sum_{j=0}^{i-1} \binom{h}{j} \eta^{h-j} e_{i-j} \\
& \quad =
\sum_{i=1}^a x_i \sum_{j=0}^{i-1} \binom{h}{j} \eta^{h-j} e_{i-j} \\
& \quad =
\sum_{i=1}^a x_i \sum_{\ell=1}^{i} \binom{h}{i - \ell} \eta^{h-i+\ell} e_{\ell} \\
& \quad =
\eta^h \sum_{\ell=1}^a 
\left[\sum_{i=\ell}^a x_i 
\frac{h(h-1)\cdots(h-(i - \ell))}{(i-\ell)!} \eta^{-(i-\ell)}\right] 
e_\ell. \\
\end{split}
\]

Write $\langle \cdot, \cdot \rangle$ and $\| \cdot \|$ for the
usual Euclidean inner product and norm on $\bR^a$, and
set $V_\rho := \{x \in \bR^a : |\langle x, e_a \rangle| \le \rho \|x\|\}$,
for $0 < \rho < 1$.    Note that the sets $V_\rho$
have the properties 
\begin{itemize}
\item
$V_{\rho'} \subset V_{\rho''}$ for $0 < \rho' < \rho'' < 1$,
\item
$\bigcup_{0 < \rho < 1} V_\rho = \bR^a \setminus \{t e_a : t \ne 0\}$,
\item
$\bigcap_{0 < \rho < 1} V_\rho = \{x \in \bR^a : \langle x, e_a \rangle = 0\}$.
\end{itemize}
In particular, the sets $\bR^a \setminus \bigcup_{0 < \rho < 1} V_\rho$ 
and $\bigcap_{0 < \rho < 1} V_\rho$ are both Lebesgue null.
It is not difficult to see from the above that
\[
J_a(\eta)^h V_{\rho''} \subseteq V_{\rho'},  \quad 0 < \rho' < \rho'' < 1, 
\quad \text{$h$ sufficiently large}.
\]

A similar argument shows that the conclusion of the previous paragraph
hold in case (c), provided $a \ge 2$.  If $a=1$ in case (c), then the sets 
$B_\epsilon := \{x \in \bR :|x| \le \epsilon\}$, $0 < \epsilon < \infty$, 
have the properties
\begin{itemize}
\item
$B_{\epsilon'} \subset B_{\epsilon''}$ for $0 < \epsilon' < \epsilon'' < \infty$,
\item
$\bigcup_{0 < \epsilon < \infty} B_\epsilon = \bR$,
\item
$\bigcap_{0 < \epsilon < \infty} B_\epsilon = \{ 0\}$.
\end{itemize}
In particular, the sets $\bR \setminus \bigcup_{0 < \epsilon < \infty} B_\epsilon$ 
and $\bigcap_{0 < \epsilon < \infty} B_\epsilon$ are both Lebesgue null.
It is clear that
\[
J_1(\eta)^h B_{\epsilon''} \subseteq B_{\epsilon'},  \quad 0 < \epsilon' < \epsilon'' < \infty, 
\quad \text{$h$ sufficiently large}.
\]

Analogous constructions for the cases (b) and (d) 
and  some straightforward
further argument complete the proof of the following result.

\begin{Lemma} 
\label{L:shrinking_sets}
Suppose that the matrix $A \in \Gamma$ is such that the
cyclic group $\{A^h : h \in \bZ\}$ does not have a compact closure.  
Then, there exists
a collection  $(D_t)_{0 < t < \infty}$ of closed subsets of
$\bR^d$ with the following properties:
\begin{itemize}
\item
$D_{t'} \subset D_{t''}$ for $0 < t' < t'' < \infty$,
\item
$\Lambda\left(\bR^d \setminus \bigcup_{0 < t < \infty} D_t\right) = 0$,
\item
$\Lambda\left(\bigcap_{0 < t < \infty} D_t\right) = 0$,
\item
$A^h D_{t''} \subseteq D_{t'}$ for $0 < t' < t'' < \infty$ and $h$ sufficiently large.
\end{itemize}
\end{Lemma}

\bigskip
\noindent
{\bf Acknowledgment:} We thank Omer Angel for giving a talk that sparked
our interest in the problem we investigate in this paper.

\providecommand{\bysame}{\leavevmode\hbox to3em{\hrulefill}\thinspace}
\providecommand{\MR}{\relax\ifhmode\unskip\space\fi MR }
\providecommand{\MRhref}[2]{%
  \href{http://www.ams.org/mathscinet-getitem?mr=#1}{#2}
}
\providecommand{\href}[2]{#2}

\end{document}